\documentclass{amsart}

\usepackage{amssymb, amsthm, amsfonts, latexsym, amsmath, dsfont, pstricks-add, mathrsfs, lmodern, nicefrac, graphicx}

\newcommand{\eps}{\varepsilon}
\newcommand{\prb}[1]{\mathbb{P}(#1)}
\newcommand{\prbb}[1]{\mathbb{P}\big(#1\big)}
\newcommand{\mean}[2][]{\mathbb{E}_{#1}(#2)}
\newcommand{\meanb}[2][]{\mathbb{E}_{#1}\big(#2\big)}
\newcommand{\meanx}[2][]{\mathbb{E}_{#1}\bigg(#2\bigg)}
\newcommand{\bfrac}[2]{\bigg(\frac{#1}{#2}\bigg)}
\newcommand{\hier}[3]{#1^{(#2)}_{#3}}
\newcommand{\ee}{\mathrm{e}}
\newcommand{\ie}{i.e.\ }
\newcommand{\eg}{e.g.\ }
\newcommand{\abs}[1]{\lvert#1\rvert}
\newcommand{\lemm}[1]{Lemma~\ref{#1}}
\newcommand{\thmm}[1]{Theorem~\ref{#1}}
\newtheorem{theorem}{Theorem}
\newtheorem{lemma}[theorem]{Lemma}
\theoremstyle{definition}
\newtheorem{definition}[theorem]{Definition}

\newtheorem*{remark}{Remark}

\title{Reaching consensus on a connected graph}
\author{John~Haslegrave}
\author{Mate~Puljiz}

\begin{document}

\begin{abstract}
We study a simple random process in which vertices of a connected graph reach consensus through pairwise interactions. We compute outcome probabilities, which do not depend on the graph structure, and consider the expected time until a consensus is reached. In some cases we are able to show that this is minimised by $K_n$. We prove an upper bound for the case $p=0$ and give a family of graphs which asymptotically achieve this bound. In order to obtain the mean of the waiting time we also study a gambler's ruin process with delays. We give the mean absorption time and prove that it monotonically increases with $p\in[0,1/2]$ for symmetric delays.


\textit{AMS: 60G40; 60G50; 60K35.}
\end{abstract}

\maketitle

\section{Introduction}
We consider the evolution of a system on a connected graph $G$ with $n$ vertices. Each vertex has a strategy taken from $\{1,\ldots, m\}$ (we will frequently write $[m]$ for this set). The starting strategies of the vertices are chosen independently and uniformly at random. At each time step an edge is chosen uniformly at random, and both vertices are updated to have the same strategy, which is the higher of the two with probability $p$ and the lower with probability $1-p$. This simple model covers a broad range of real-life scenarios where a consensus is reached via pairwise interactions among the individual agents, whether we are interested in modelling an infectious disease spread or the process by which a certain gene became prevalent in the human genome.

The model was inspired by the well-studied tournament games in the theory of genetic algorithms (see Rowe, Vose and Wright \cite{DiffCG, PreDiffCG} and Vose \cite{SGAVose}) and indeed, it is a generalisation of these as it is easily seen that they reduce to the complete graph instance of our problem. The idea is that the underlying connected graph allows for modelling a spatial aspect of the problem at hand. It therefore comes as a surprise when in Section \ref{Graphs} we prove that the probability that a certain strategy prevails does not depend on the network structure of the nodes. This is achieved by reducing the problem to the study of the two-strategy case by looking at contiguous partitions of the strategy set. Validity of these coarse grainings was previously checked only for the complete graph case in \cite{PreDiffCG}.

The model resembles the \textit{voter model}, introduced as a lattice model by Clifford and Sudbury \cite{CS73} and adapted to more general graphs by Donnelly and Welsh \cite{DW83}. This is a continuous-time process in which each vertex adopts the strategy of a randomly-chosen neighbour at rate 1. There are two principal differences in our model. First, each update is given by a randomly chosen edge, not a randomly-chosen vertex; the two are equivalent only in the special case of regular graphs. Secondly, the voter model does not distinguish between strategies, whereas our model accounts for the possibility that some strategies are more effective than others. Donnelly and Welsh consider how the underlying graph may be chosen to minimise or maximise the expected time to reach a consensus, and Hassin and Peleg \cite{HP01} consider the same problem for a related discrete-time process with synchronous updates, where every vertex simultaneously adopts the strategy of a randomly-chosen neighbour at each time step. This latter process has the same completion time as a system of $n$ coalescing random walks starting at the vertices of $G$, and Cooper, Els\"asser, Ono and Radzik \cite{CEOR} recently gave improved bounds for this coalescence time. Both the continuous-time and the synchronous discrete-time voter models have an expected $n$ vertices updating in every unit of time, so the bounds on these models must be multiplied by $n$ for a meaningful comparison with our model, where only one vertex updates at each time step.

The expected time to reach a consensus will, of course, depend on the graph structure and we are able to give explicit formula for this mean only for the case of the complete graph and two strategies. This is done in Section \ref{GR} by relating the problem to a version of the gambler's ruin process with delays. By elementary means we show that this expression is monotonic in parameter $p\in[0,1/2]$ in the case of symmetric delays, which translates to a monotonicity result for the expected decision time of our process on the complete graph. Computer simulations using the PRISM model-checking software \cite{PRISM} seem to indicate that this holds true more generally for any fixed connected graph but the proof remains elusive.

It seems natural to conjecture that for a fixed parameter $p$ the process over the complete graph, on average, reaches consensus most quickly. This is again supported by the computer simulations but we are only able to prove it in the class of regular graphs where each node is adjacent to the same number of neighbours and with the restriction to two strategies, see Section \ref{Graphs}. This proves a conjecture of Donnelly and Welsh \cite{DW83}. It is less clear which graph we might expect to be slowest to reach consensus, and in fact PRISM simulations suggest the answer depends on $p$. For the case $p=0$ we give in Section \ref{CC} good bounds for the expected time for some specific types of graph, and an upper bound on the time taken for any graph, together with a family of graphs which asymptotically attain this bound within a small error term. These results make use of some generalisations of the coupon collector's problem.

\section{Absorption probabilities and the optimality of the complete graph}\label{Graphs}

Since $G$ is connected, eventually the process will, with probability 1, reach a state where only one strategy remains. Write $S$ for the strategy that is left; we give the precise distribution of $S$ in terms of $n$, $m$ and $p$. Note that this distribution does not depend on the structure of $G$, only its order.

\begin{theorem}For any graph $G$ with $n$ vertices, if the initial state is chosen uniformly at random from $[m]^n$ then $\prb{S=l}=1/m$ if $p=1/2$, and
\[\prb{S=l}=\frac{(l-2lp+mp)^n-(l-1-2lp+2p+mp)^n}{(m-mp)^n-(mp)^n}\]
otherwise.
\end{theorem}
\begin{proof}
Trivially if $p=1/2$ all strategies are equivalent, and each is equally likely to remain to the end, so we may assume $p\neq 1/2$. We first compute $\prb{S\leqslant l}$ by coarse-graining the strategies into those at most $l$ and those exceeding $l$; call these sets of strategies $A$ and $B$ respectively. This is a coarse graining in the sense that when we consider the vertices as playing strategies in $\{A,B\}$ nothing changes unless the edge chosen has one vertex with strategy $A$ and one with strategy $B$ (call this a ``significant edge''), in which case they will both adopt $B$ with probability $p$ and both adopt $A$ with probability $1-p$. Thus the coarse-grained process we obtain on strategies $\{A,B\}$ is exactly the same as the original process for $m=2$, save that the distribution of starting states is different. We will have $S\leqslant l$ if and only if the coarse-grained process reaches consensus with all vertices playing $A$.

Write $a_0$ for the number of vertices initially playing $A$, and let $a_r$ be the number playing $A$ after the $r$th time a significant edge is chosen. The evolution of $a_r$ is a random walk with absorbing states at $0$ and $n$, where $a_r=a_{r-1}+1$ with probability $p$ and $a_r=a_{r-1}-1$ with probability $1-p$, independent of which edges are chosen, and indeed independent of $G$. So the probability that $a_r$ reaches $n$ before $0$, \ie $\prb{S\leqslant l}$, does not depend on $G$, only on $n$, $m$ and $l$. Therefore $\prb{S=l}=\prb{S\leqslant l}-\prb{S\leqslant l-1}$ is also independent of $G$.

Note that, since $p\neq 1/2$, the sequence $\big(\frac{1-p}{p}\big)^{a_r}$ is a bounded martingale. Writing $T$ for the value of $r$ at which the random walk stops, $T$ is a stopping time with finite expectation and so, by the Optional Stopping Theorem (see \eg \cite{williams1991probability}, p.\ 100)
\[
\meanx{\bfrac{1-p}{p}^{a_T}\Bigm|a_0=a}=\bfrac{1-p}{p}^a\,.
\]
It follows that
\begin{align*}
\prbb{a_T=n\bigm|a_0=a}&=\frac{\big(\frac{1-p}{p}\big)^a-1}{\big(\frac{1-p}{p}\big)^n-1} \\
&=\frac{(1-p)^ap^{n-a}-p^n}{(1-p)^n-p^n}\,.
\end{align*}
Since $a_0$ is distributed as $\operatorname{Bin}(n,l/m)$, we have
\begin{align*}
\prb{S\leqslant l}&=\sum_{a=0}^n\binom{n}{a}\bfrac{l}{m}^a\bfrac{m-l}{m}^{n-a}\bfrac{(1-p)^ap^{n-a}-p^n}{(1-p)^n-p^n} \\
&=\frac{(l-2lp+mp)^n-(mp)^n}{(m-mp)^n-(mp)^n}\,,
\end{align*}
and so
\[
\prb{S=l}=\frac{(l-2lp+mp)^n-(l-1-2lp+2p+mp)^n}{(m-mp)^n-(mp)^n}\,,
\]
as required.
\end{proof}

The probability of a particular strategy remaining at the end does not depend on the structure of $G$, but the time taken until this point is reached will do. It is natural to conjecture that the graph which has the quickest expected time is $K_n$. We prove this for the special case where $G$ is known to be regular (that is, having all degrees equal) and $m=2$.

\begin{theorem}\label{complete}For $m=2$ and any values of $n$ and $p$, $K_n$ has the shortest expected time to completion of any $n$-vertex regular graph.
\end{theorem}
\begin{proof}
Let $a_i$ be the number of vertices with strategy 1 after $i$ significant edges have been chosen, and let $t_i$ be the time between choosing the $(i-1)$th and $i$th significant edges. We will show that, for any fixed sequence $(a_i)$, $\mean[G]{t_i\mid a_0,\ldots,a_i}\geqslant \mean[K_n]{t_i\mid a_0,\ldots,a_i}$; the result follows by averaging.

Let $G$ be $k$-regular and $a_0,\ldots,a_r$ be a fixed possible sequence (i.e. $a_r=0$ or $a_r=n$, $0<a_i<n$ for $i<r$ and $a_i-a_{i-1}=\pm 1$ for each $i>0$). Note that the probability of a given sequence depends only on the number of increments and decrements, and not on $G$. Let $E_i$ be the number of significant edges at time $\sum_{j\leqslant i}t_j$ (when there are $a_i$ vertices with strategy 1). Then $\mean{t_{i+1}\mid E_i}=kn/(2E_i)$, so $\mean{t_{i+1}}=\mean{\mean{t_{i+1}\mid E_i}}=\mean{kn/(2E_i)}\geqslant kn/(2\mean{E_i})$ by Jensen's inequality. Since $\mean[K_n]{t_{i+1}}=n(n-1)/(2a_i(n-a_i))$, it is sufficient to prove that $\mean[G]{E_i}\leqslant ka_i(n-a_i)/(n-1)$ for any $k$-regular graph $G$. We prove this by induction on $i$; it is true for $i=0$ since each edge has probability $a_0(n-a_0)/\binom{n}{2}$ of being significant, and there are $kn/2$ edges.

Suppose that the result holds for $i$ and assume that $a_{i+1}=a_i-1$ (the case $a_{i+1}=a_i+1$ is similar). Then write $v_1,\ldots, v_{a_i}$ for the vertices playing strategy 1 and $e_1,\ldots,e_{a_i}$ for the numbers of significant edges meeting them. The probability that the next significant edge to be sampled meets $v_j$ is $e_j/E_i$, and if it does then $e_j$ edges become non-significant and $k-e_j$ edges become significant, so $E_{i+1}=E_i+k-2e_j$. Since $E_i=\sum_{j=1}^{a_i} e_j$,
\begin{align*}
\mean{E_{i+1}\mid e_1,\ldots,e_{a_i}}&=\sum e_j+k-2\frac{\sum e_j^2}{\sum e_j} \\
&\leqslant (1-2/a_i)E_i+k \,.
\end{align*}
If $a_i=1$ then $a_{i+1}=E_{i+1}=0$. Otherwise $1-2/a_i\geqslant 0$ and so
\begin{align*}
\mean{E_{i+1}}&=\meanb{\mean{E_{i+1}\mid e_1,\ldots,e_{a_i}}} \\
&\leqslant (1-2/a_i)\mean{E_i}+k \\
&\leqslant (1-2/a_i)ka_i(n-a_i)/(n-1)+k \\
&=\frac{k}{n-1}((n-a_i)(a_i-2)+n-1) \\
&=k(a_i-1)(n-(a_i-1))/(n-1) \,,
\end{align*} 
as required.
\end{proof}

\thmm{complete} does not immediately give the same result for larger values of $m$. While we know, from the coarse-graining argument, that for \eg $m=3$ the expected time for either $\{1\}$ or $\{2,3\}$ to be eliminated and the expected time for either $\{1,2\}$ or $\{3\}$ to be eliminated are both minimised by $K_n$, it does not follow that the expectation of the maximum of these two times is also minimised by $K_n$.

Since for $p=1/2$ and $G$ regular, the model is equivalent to the voter model, \thmm{complete} shows as a special case that the complete graph minimises the time taken for the voter model among regular graphs, as conjectured by Donnelly and Welsh \cite{DW83}.

\section{Upper bounds on the time to completion}\label{CC}

In this section we consider which graphs give the longest expected time to completion. We only consider the special case $p=0$; the authors used the PRISM model-checking software \cite{PRISM} to analyse the expected times for general $p$ on a variety of graphs, and the results suggest that the answer is different for larger $p$. While it is natural to expect that if the complete graph is fastest, a sparse graph such as the path might be slowest, this is not the case. In fact, since the average time until a specific edge is sampled is equal to the number of edges in the graph, the slowest cases are graphs which are both sparse in parts (so that specific edges may be needed) and dense in parts (so that it takes a long time to sample the necessary edges). We prove a general bound on the expected time for the process with $p=0$ on any graph, but also consider some particular graphs. As well as the path, we consider three natural families of graphs with the property of being sparse in parts and dense in others. The \textit{sundew} consists of a clique (that is, a complete subgraph) with some pendant edges attached as evenly as possible; the \textit{lollipop} consists of a clique with a pendant path, and the \textit{jellyfish} is something of a hybrid between the two, consisting of a clique with several shorter pendant paths attached as evenly as possible. 
PRISM simulations suggest that for small $n$ a sundew is the slowest graph for $p$ close to 0, with a lollipop being slowest for $p$ close to $1/2$. Our theoretical results on the special case $p=0$ indicate that suitably-chosen jellyfish are slower if $n$ is sufficiently large, but that the sundew on $n$ vertices with a clique of size $n-r$ remains slower than the lollipop with the same parameters. In fact the sundew is a special case of the \textit{spider graph}, defined by Donnelly and Welsh \cite{DW83} to be a clique with pendant edges attached in any manner; likewise the lollipop may be regarded as a generalisation of their \textit{tennis-racquet graph}, which is a clique with a single pendant edge. 

In our analysis, we consider some variants of the coupon collector problem. The classical setting, in which we collect coupons which are independently equally likely to be any of $n$ types, and ask for the expected time until we have at least one of each type, is a folklore result (see \eg \cite{Feller}); the answer is $nH_n$, where $H_n$ is the $n$th harmonic number, so is $n\log n+O(n)$. The variant known as the double dixie cup problem asks for the time until $m$ copies of each coupon have been collected, for fixed $m$, and is rather harder. Newman and Shepp \cite{NS60} gave the expectation $n(\log n+(m-1)\log\log n+O(1))$, and Erd\H{o}s and R\'enyi \cite{ER61} studied the distribution in more detail. More recently, other aspects of the problem have been studied by Myers and Wilf \cite{MW06}. Here we consider a setting in which different types have different targets, which are themselves random variables; choosing geometric random variables gives a particularly simple relationship with the original coupon collector problem. We will also need an inequality satisfied by processes which are similar to the classical coupon collector except for allowing the possibility of receiving multiple coupons in a single time step, which we prove below. Throughout, we will use the term ``geometric random variable'' to mean a variable equivalent to the number of independent identical trials up to and including the first success, \ie the geometric distribution has support $\{1,2,\ldots\}$.

First we consider the relaxation of the classical problem in which multiple coupons may be received simultaneously. We write $h(n)$ for the function with $h(n)=H_n$ if $n\in\mathbb{N}_0$ which is linear in between these points; note that $h(n)$ is concave.

\begin{lemma}\label{multipass}Suppose we have a process in which the probability of receiving a coupon of type $i$ is $1/N$ at each time step for every $i\in[n]$, where $N\geq n$. Write $T$ for the time at which we first have at least one coupon of every type. Then $\mean{T}\leq NH_n$, with equality if and only if the probability of receiving more than one new coupon simultaneously at any point is 0.
\end{lemma}
\begin{proof}We prove this by induction on $n$; it is trivial for $n=0$. Write $A_t$ for the number of new types received at time $t$, and $B_t$ for the number of coupons received before time $t$. Then $\mean{A_t}=(n-B_t)/N$, so $\sum_t(A_t+(B_t-n)/N)$ is a martingale. Run the process until the first time one or more coupons are received; write $T_1$ for the time at which this occurs. $T_1$ is a stopping time and $\mean{T_1}\leq N<\infty$. The martingale has bounded variation, so the Optional Stopping Theorem (see \eg \cite{williams1991probability}, p.\ 100) applies and so 
\[
0=\meanx{\sum_{t=1}^{T_1}\bigg(A_t+\frac{B_t-n}{N}\bigg)}\,.
\]
Since $B_t=0$ for $t\leq T_1$ and $A_t=0$ for $t<T_1$, the above gives $\mean{T_1}=N\mean{A_{T_1}}/n$. Now we have 
\begin{align*}
\mean{T}&\leq\mean{T_1}+N\mean{h(n-A_{T_1})}\\
&=N(\mean{A_{T_1}}/n+\mean{h(n-A_{T_1})})\\
&\leq N(\mean{A_{T_1}}/n+h(n-\mean{A_{T_1}}))\,,
\end{align*}
by the induction hypothesis and Jensen's inequality. Since $h(n-1)+1/n=h(n)$, and $h'(x)>1/n$ for $x<n-1$, we also have $a/n+h(n-a)\leq h(n)$ if $a\geq 1$, with equality if and only if $a=1$. Setting $a=\mean{A_{T_1}}$, and noting that $A_{T_1}\geq 1$ by definition, we get $\mean{T}\leq Nh(n)$ with equality if and only if both $A_{T_1}\equiv 1$ and the process after time $T_1$ also satisfies the condition that the probability of receiving multiple new coupons simultaneously is 0 (note that the first condition implies equality in the Jensen's inequality step).
\end{proof}

We next consider a process in which we potentially require multiple coupons of each type, with each type having a target number given by a geometric random variable. In general these variables will not be independent, and so we first define the types of dependencies we permit.

\begin{definition}Let $(X_i)_{i\geq 1}$ be a sequence of independent Bernoulli random variables with parameter $q$. For each $j\in[n]$, let $(\hier{i}{j}{k})_{k\geq 1}$ be a sequence of positive integers such that $\hier{i}{j}{k}\neq \hier{i}{j}{l}$ whenever $k\neq l$. For each $j\in [n]$, let $Y_j=\min\{k:X_{\hier{i}{j}{k}}=1\}$. A \textit{system of connected geometrics} is a set of random variables $(Y_j)_{j\in[n]}$ produced by such a construction.\end{definition}

We now prove an upper bound on the expected time taken by a collecting process with targets given by such a system of variables.

\begin{lemma}\label{coupon}Consider a process where we receive a coupon of type $i$ at each time step with probability $1/N$ for each $i\in[n]$, where $N\geq n$. For each $i\in[n]$ we require $Y_i$ coupons of type $i$, where the $Y_i$ are a system of connected geometrics with parameter $q$. Then the expected time to completion is at most $q^{-1}NH_n$, with equality if both the $Y_i$ are independent and the probability of receiving two or more coupons at the same time is $0$.
\end{lemma}
\begin{proof}If the $Y_i$ are independent, \ie the sets $\{\hier{i}{j}{k}\mid k\geq 1\}$ and $\{\hier{i}{j'}{k}\mid k\geq 1\}$ are disjoint whenever $j\neq j'$, then the process takes the same time as a process where each type is received with probability $q/N$ and only one coupon of each type is required. To see this, consider the same process except that, instead of revealing the $X_i$ initially to determine the $Y_j$, every time we receive a coupon of type $j$ we reveal the next variable in the sequence $X_{\hier{i}{j}{k}}$, and if it is 0 we discard that coupon. In this process we keep a coupon of any given type at each time step with probability $q/N$, and finish when we have kept a coupon of every type. The expected time for this is at most $q^{-1}NH_n$ by \lemm{multipass}, with equality if no two coupons can be received simultaneously.

Next we show that adding dependencies between the $Y_i$ only decreases the expected time. Suppose that we have two processes: process A where $X_i$ occurs in two or more sequences, and process B where one occurrence of $X_i$ (say for coupon type $j$) is replaced by a new variable $X$ with the same distribution, but which is otherwise identical to process A. Fix the values of all variables except $X_i$ and $X$. We may assume that the two processes require different numbers of type-$j$ coupons, $k$ and $l$ with $l>k$ (and consequently that $X\neq X_i$), since otherwise they finish at the same time. Write $t_0$ for the time at which, if $X_i=0$, both processes will have collected enough coupons of all types other than possibly type $j$; write $t_1$ for the time at which this happens if $X_i=1$, and write $t$ for the time at which both processes will have received $l$ coupons of type $j$. 
If $X_i=0$ and $X=1$ then process A takes $\min\{0,t-t_0\}$ steps longer than process B, whereas if $X_i=1$ and $X=0$ then B takes $\min\{0,t-t_1\}$ steps longer than A. These two events have the same probability, and clearly $t_1\leq t_0$ since increasing $X_i$ cannot increase $Y_{j'}$ for any $j'\neq j$ (and changing $X$ does not affect $Y_{j'}$ at all). So the expected time for B is at least that for A.

Now, start from a system of independent geometrics and add dependencies one by one. The expected time decreases at each stage. If there are only finitely many dependencies to add, then the final expected time will be at most $q^{-1}NH_n$. If there are infinitely many, the expected times after finitely many dependencies have been added converge to that of the final process, since two processes which have the same values $\hier{i}{j}{k}$ for $k<K$ will have expected times which are close together, for $K$ sufficiently large.
\end{proof}

Next we consider the case where one type is less likely to occur than any other.

\begin{lemma}\label{pathends}Suppose we run two processes where at each time step we receive a single coupon with probability $n/N$, equally likely to be any of $n$ types. If we receive a coupon of type other than type 1, we keep it with probability $q$. In one process we also keep coupons of type 1 with probability $q$, but in the other we keep them with probability $q'$, where $q$ and $q'$ are fixed with $0<q'<q<1$. Then the expected times for the two processes differ by $o(N)$.
\end{lemma}
\begin{proof}Couple the two processes, so that the same type of coupon is received at each step and either both coupons are kept or a coupon of type 1 is kept in the first process only. The second process takes longer than the first only if the last type kept in the second process is type 1. This happens with probability $\prod_{k=1}^{n-1}\frac{kq}{kq+q'}$. Since
\[
\lim_{n\to\infty}\prod_{k=1}^{n-1}\frac{kq}{kq+q'}\leq\lim_{n\to\infty}\exp\bigg(-\sum_{k=1}^{n-1}\frac{q'}{kq+q'}\bigg)=0\,,
\]
with high probability the two processes take the same time. When they do not, the difference is the time taken to keep a coupon of type 1, which has expected value $Nq'^{-1}$. Consequently the difference in expected times is $o(N)$, as required, and so the time for the second process is $N(q^{-1}H_n+o(1))$.
\end{proof}

We can regard the second process as equivalent to one where instead of having the same chance of receiving each coupon, but a reduced chance of keeping one type, we have a reduced chance of receiving that type in the first place. Consequently, by \lemm{coupon}, if each type has a target given by a system of connected geometrics, with type 1 being received with probability $1/aN$ and other types being received with probability $1/N$ each, then the expected time is at most $N(q^{-1}H_n+o(1))$. Exactly the same argument applies when there are two types (or any constant number) which have the lower probability of being received.

We now return to our original process running on a connected graph $G$. We consider the case $p=0$, $m=2$, which we may think of as having vertices either active or inactive. Active vertices never change their status, while an inactive vertex becomes active when an edge to a neighbouring active vertex is sampled by the process. For $U\subseteq V(G)$, write $T_U$ for the time until all vertices in $U$ are active, setting $T_U=0$ if all vertices in $G$ are inactive at the starting point. In the following analysis, we sometimes consider the slightly different setting in which the starting state is chosen uniformly at random among states with at least one active vertex. The two expectations differ by a factor of $1-2^{-n}$ (since only one state of a possible $2^n$ is excluded), which is much smaller than the error terms in our estimates.  We are now ready to prove an upper bound on the time taken for this process to reach consensus.

\begin{theorem}\label{upperbound}For any connected graph $G$ with $n$ vertices, $\meanb{T_{V(G)}}<n^2\log n+n$.
\end{theorem}
\begin{proof}For each vertex $v$ in turn, define a sequence $\hier{u}{v}{i}$ such that $\hier{u}{v}{1}=v$, for every $i>1$ $\hier{u}{v}{i}$ is adjacent to some vertex $\hier{u}{v}{j}$ with $j<i$, and $\big\{\hier{u}{v}{i}\mid i\in[n]\big\}=V(G)$. Write $d_v$ for the minimum distance from $v$ to a different active vertex, artificially setting $d_v=n$ if there are no other active vertices in $V(G)$. $d(v,\hier{u}{v}{i})\leq i-1$, so $d_v$ is bounded by $\min\{i>1:\hier{u}{v}{i}\text{ active}\}-1$. These bounds (for each different $v$) form a system of connected geometrics $(Y_v)_{v\in V(G)}$. Suppose we run the process on $G$, recording at every time step whether the edge sampled reduces the distance from $v$ to the nearest active vertex. Then there is at least a probability of $1/e(G)$ of this happening at each time step while $v$ is inactive; once $v$ becomes active we can make false records with probability $1/e(G)$ and it will make no difference. Once the number of records for $v$ reaches $Y_v$, $v$ must be active. This process dominates a collecting process where each of $n$ types of coupon is received with probability exactly $1/e(G)$ and the targets form a system of connected geometrics with parameter $1/2$, which by \lemm{coupon} takes time at most $2e(G)H_n<n(n-1)(\log n+1)<n^2\log n+n$, as required.\end{proof}

This bound is close to best possible, as we will show by considering jellyfish graphs. Before analysing the process on sundews, lollipops and jellyfish in more detail we first prove some lemmas on the time until a subset of vertices of a particular type reaches its final state. 

\begin{lemma}\label{leaves}Write $L$ for the set of vertices of degree 1 in $G$, and suppose $\abs{L}=r$. Then as $r\to\infty$,
\[e(G)(h(r)-\log 2+o(1))\leq\mean{T_L}\]
and
\[\meanb{T_{V(G)}}\leq e(G)(h(r)-\log 2+o(1))+\meanb{T_{V(G)\setminus L}}\,.\]
\end{lemma}
\begin{proof}Write $L_0$ for the set of vertices in $L$ which start off inactive. 
For the upper bound, once all vertices outside $L$ are active (after time $T_{V(G)\setminus L}$) it is sufficient to sample each edge leading to $L_0$ once. 
The additional time taken for this to happen is $e(G)h(\abs{L_0})$ for a particular $L_0$, so 
$\meanb{T_{V(G)}-T_{V(G)\setminus L}}\leq e(G)\mean{h(\abs{L_0})}$. By Jensen's inequality this is at most $e(G)h(\mean{\abs{L_0}})=e(G)h(r/2)$. Since $h(r)=\log r+\gamma+o(1)$, where $\gamma$ is Euler's constant, $h(r/2)=h(r)-\log 2+o(1)$, and so
\[\meanb{T_{V(G)}}\leq e(G)(h(r)-\log 2+o(1))+\meanb{T_{V(G)\setminus L}}\,.\]

For the lower bound, note that each of the $\abs{L_0}$ edges meeting $L_0$ must be sampled for all vertices in $L$ to become active. Each edge has probability $1/e(G)$ to be sampled at each time step, and only one can be sampled at any time step, so this takes time $e(G)h(\abs{L_0})$, by \lemm{multipass}. Consequently $\mean{T_L}\geq e(G)\mean{h(\abs{L_0})}$. Fix $\eps>0$; with probability at least $1-\ee^{-2\eps^2r}$, by Hoeffding's inequality, $\abs{L_0}\leq(1/2-\eps)r$. So
\begin{align*}
\mean{h(\abs{L_0})}&>(1-\ee^{-2\eps^2r})h((1/2-\eps)r)\\
&=h(r)+\log(1/2-\eps)+o(1)\,.
\end{align*}
Given $\delta>0$ choose $\eps$ such that $\log(1/2-\eps)<-\log 2-\delta$; then for large $r$ we have $\mean{h(\abs{L_0})}>h(r)-\log 2-\delta$, so $\mean{T_L}\geq e(G)(h(r)-\log 2+o(1))$, as required.\end{proof}

\begin{lemma}\label{clique}Let $S$ be an $r$-clique in $G$. Then $\mean{T_S}=e(G)o(1)$ as $r\to\infty$.\end{lemma}
\begin{proof}First we bound the expected time until some vertex of $S$ is active (assuming some vertex in $G$ was initially active). With probability $1-2^{-r}$ this time is 0; if not the distance, $d$, from $S$ to the nearest active vertex is bounded by a geometric variable with rate $1/2$. There are $d$ edges which, if sampled in turn, will lead to a vertex in $S$ being active, and the expected time to sample these edges in turn is $2e(G)$. So the overall expected time until a vertex in $S$ is active is at most $2^{1-r}e(G)=o(1)e(G)$.

Secondly we bound the time from any position with at least one active vertex in $S$ until all vertices in $S$ are active. It is sufficient to deal with the case where exactly one vertex is active, since additional active vertices can only reduce the expected time. If $k$ vertices in $S$ are active, there are at least $k(r-k)$ edges which, if sampled, will increase the number of active vertices to $k+1$. Thus the expected time until all vertices in $S$ are active is at most $e(G)\sum_{k=1}^{r-1}\frac{1}{k(r-k)}$. Since $\frac 1{k(r-k)}=\big(\frac 1r\big)\big(\frac 1k+\frac 1{r-k}\big)$, this bound equals $e(G)(2H_{r-1}/r)=e(G)o(1)$. Combining these two estimates, $\mean{T_S}=e(G)o(1)$, as required.
\end{proof}

\begin{lemma}\label{path}Suppose $G$ contains $r$ vertices of degree 1 or 2 arranged in a path, $P$. Then $\mean{T_P}=e(G)(h(r)-\log 4+o(1))$ as $r\to\infty$.\end{lemma}
\begin{proof}Consider the intervals of inactive vertices along the path, starting from one end. Each one in turn has its initial length dominated by independent geometric variables with parameter $1/2$ (dominated by rather than equal to since the total length is capped). With high probability there are fewer than $r/4+r^{2/3}$ such intervals (and there are at most $r/2$). Each one, except possibly the first and last, has active vertices at both ends.

Run a collection process with targets given by $r/4+r^{2/3}$ independent geometrics, with each type having probability $2/e(G)$ of being received at every time step except for two which have probability $1/e(G)$. Couple it to the process on $G$ by ensuring that each interval corresponds to one of the targets (or is dominated by it) and that while an inactive interval still exists a coupon of the corresponding type is received exactly when an edge at one end of the interval is sampled. By \lemm{coupon} and \lemm{pathends}, this takes time at most $2e(G)(h(r/4)+o(1))/2=e(G)(h(r)-\log 4+o(1))$.

For the lower bound, note that with high probability the process is dominated by a collecting process with $r/4-r^{2/3}$ independent geometric targets for which each type has probability $2/e(G)$ of being received. This has expected time $2e(G)h(r/4-r^{2/3})/2=e(G)(h(r)-\log 4+o(1))$.
\end{proof}

As a consequence of \lemm{path}, the expected time taken by the process on $P_n$ is $n\log n-O(n)$. We are now ready to compare the sundew and lollipop. In fact our result applies to any spider graph, not just the sundew, but simulations suggest that the expected time is longer on the sundew than on other spider graphs.

\begin{theorem}Let $\mathrm{Sd}_{n,r}$ be the sundew with a clique of size $n-r$ and $r$ pendant edges; let $\mathrm{Lp}_{n,r}$ be the lollipop with a clique of size $n-r$ and a pendant path of length $r$. Then, provided both $r$ and $n-r$ tend to infinity, $\mathrm{Sd}_{n,r}$ has expected time $e(\mathrm{Sd}_{n,r})(h(r)-\log 2+o(1))$ whereas $\mathrm{Lp}_{n,r}$ has expected time $e(\mathrm{Lp}_{n,r})(h(r)-\log 4+o(1))$. In particular, since $e(\mathrm{Sd}_{n,r})=e(\mathrm{Lp}_{n,r})$, the sundew has the higher expected time if $n-r$ and $r$ are both large.
\end{theorem}
\begin{proof}By \lemm{leaves}, the expected time for the sundew is at least $e(\mathrm{Sd}_{n,r})(h(r)-\log 2+o(1))$ and at most $e(\mathrm{Sd}_{n,r})(h(r)-\log 2+o(1))+\mean{T_S}$, where $S$ is the set of vertices in the clique. But $\mean{T_S}=o(e(\mathrm{Sd}_{n,r}))$ by \lemm{clique}, giving the required result.

Similarly, for the lollipop we have a lower bound of $e(\mathrm{Lp}_{n,r})(h(r)-\log 4+o(1))$ for the path to become all active, by \lemm{path}, and an upper bound of $e(\mathrm{Lp}_{n,r})(h(r)-\log 4+o(1))+\mean{T_S}$. Again $\mean{T_S}=o(e(\mathrm{Lp}_{n,r}))$ by \lemm{clique}, giving the required result.
\end{proof}

By choosing $r$ so that $h(r)/h(n)\to 1$ but $r/n\to 0$, \eg $r=n/\log n$, we construct two sequences of graphs with expected time $(1/2-o(1))n^2\log n$, about half the bound in \thmm{upperbound}. However, the jellyfish construction can do better, equalling the bound up to a factor of $1-o(1)$.

\begin{theorem}Let $J_n$ be the jellyfish graph consisting of a clique of size $n-2n/\log_2 n$, with $n/(\log_2 n)^2$ pendant paths of length $2\log_2 n$ each. Then the expected time of the process on $J_n$ is $(1-o(1))n^2\log n$.
\end{theorem}
\begin{proof}The upper bound follows from \thmm{upperbound}. For the lower bound, we show that the expected time until the end of every path is active is at least this long. The expected number of paths which start off all inactive is $(n/(\log_2 n)^2)2^{-2\log_2 n}=1/n(\log_2 n)^2=o(1)$. So with high probability no such path exists. Consider only the paths which have an inactive vertex of degree 1 at the end, and suppose there are $k$ of these. Each of these $k$ ends has a distance to the nearest active vertex given by a geometric random variable, and these variables are independent since the paths are disjoint. Consider a process where for $i\in[k]$ we receive a coupon of type $i$ if the distance from the $i$th end to the nearest active vertex is reduced. By \lemm{coupon}, the expected time for this process is $e(G)h(k)$, since the variables are independent and at most one path contains any sampled edge. The overall expected time is therefore at least $(1-o(1))e(G)\mean{h(k)}$. Applying the Chernoff bound, with high probability $k\geq n/(4(\log_2 n)^2)$, and so the expected time is at least $(1-o(1))e(G)h(n/(4(\log_2 n)^2))$. The required lower bound follows since $h(n/(4(\log_2 n)^2))>\log n-4\log\log_2 n=(1-o(1))\log n$.
\end{proof}

\section{Gambler's ruin with delays}\label{GR}
In this section we consider a special case of our problem of reaching a consensus on the complete graph with $n$ vertices where $m=2$. Because of the symmetries of the complete graph, from the probabilistic point of view, it is easily seen that the evolution of this system is isomorphic to a random walk over the set of states $\{0,1,\dots,n\}$ with $0$ and $n$ being absorbing states. More precisely, given that $k$ vertices are active and the remaining $n-k$ are inactive, the probability of sampling a significant edge (see Section \ref{Graphs}) is $\gamma_k=\frac{2k(n-k)}{n(n-1)}$ and conditionally on choosing a significant edge the probability of activating yet another vertex is $1-p$, and with probability $p$ a previously active vertex is deactivated. We remark in passing that the probability to sample a significant edge is symmetric under swapping the strategies, $\gamma_{n-k}=\gamma_{k}$. Below, we recall some of the theory on random walks relevant to our problem.

Gambler's ruin (GR) is a classical problem in probability theory. Given fixed parameters $p\in[0,1/2]$ and $n\in\mathbb{Z}^+$, a Markov chain $(X_t)_{t\in\mathbb{N}_0}$ over the state space $\{0,1,\dots,n\}$ is defined as follows. The states $0$ and $n$ are set to be absorbing and the remaining transition probabilities for states $0<k<n$ are given by
\begin{gather*}
 p_{k,k-1}=\prb{X_{t+1}=k-1\mid X_t=k}=p,\\
 p_{k,k+1}=\prb{X_{t+1}=k+1\mid X_t=k}=1-p.
\end{gather*}
This Markov chain models the situation where a gambler enters a casino with $\pounds X_0$ in his pocket and plays a sequence of games in which his odds of winning are $p:1-p$ and each time he bets $\pounds 1$ on his win. This continues until he either hits his goal $\pounds n$, or until he bankrupts, whichever occurs first. The time of this happening is represented by the random variable $T=\min\{t\in\mathbb{N}_0 : X_t\in\{0,n\}\}$ and is usually called the absorption time, which is easily seen to be almost surely finite. 

There are a few interesting quantities to investigate in this setting: the probability of gambler's ruin and how it depends on the initial capital $\prb{X_T=0\mid X_0=k}$, the expected time $\mean{T\mid X_0=k,X_T=0}$ for this to happen, etc.
It turns out that for the classical GR many of these quantities can be explicitly computed. This is usually done by employing martingale theory (see e.g.\ Williams \cite{williams1991probability}), or, more elementary, by solving certain recurrence relations (as in \cite{Feller}).

In the present article we seek to analyse a more general problem of gambler's ruin with delays (DGR). Given $p$ and $n$ as before, and a sequence of parameters $(\gamma_1,\dots,\gamma_{n-1})\in(0,1]^{n-1}$ we define a new Markov chain $(X_t)_{t\in\mathbb{N}_0}$ over $\{0,1,\dots,n\}$ with $0$ and $n$ still being absorbing states and the following transition probabilities for $0< k < n$:
\begin{gather*}
 p_{k,k-1}=\prb{X_{t+1}=k-1\mid X_t=k}=p\gamma_k,\\
 p_{k,k+1}=\prb{X_{t+1}=k+1\mid X_t=k}=(1-p)\gamma_k,\\
 p_{k,k}=\prb{X_{t+1}=k\mid X_t=k}=1-\gamma_k.
\end{gather*}
This modifies the previous model by allowing a draw outcome of a game with probability $1-\gamma_k$ in which case our gambler's fortune is unchanged, and conditioned on winning or losing $\pounds 1$ the probabilities are the same as before. It seems artificial to allow the probability of the draw outcome to depend on the current fortune of the gambler but, for our purposes this is exactly what was needed, as the number of significant edges (and hence the probability to sample one) at any time depends only on the number of currently active vertices.

Note that setting $\gamma_0=0, \gamma_n=0$ the formulae above extend to $0\le k \le n$. In the special case $\gamma_1=\dots=\gamma_{n-1}=1$ we recover the classical GR.

There is a vast amount of literature dealing with gambler's ruin and its extensions. This ranges from classical textbooks on probability such as Feller's \cite{Feller} to recent papers generalising the original problem in various directions. Engel \cite{Engel}, Stirzaker \cite{Stirzaker, Stir2}, Bruss, Louchard and Turner \cite{NGR1}, and Swan and Bruss \cite{NGR2} all look at the problem of $N>2$ gamblers playing each other at random and compute probabilities of each player being ruined and various other associated quantities depending on the initial wealth distribution. Some authors refer to this as $N$-tower problem as the process can be visualised by $N$ towers of stacked coins where at each step a coin is taken from the top of a tower chosen at random and placed on another tower amongst the others chosen again at random. The game stops when one of the towers becomes empty. 

Other variations include two players (a casino and a gambler) with multiple currencies \cite{multiple} by Kmet and Petkov\v{s}ek. Lengyel in \cite{Tiesallowed} allows ties, and more generally Katriel in \cite{Katriel,Katriel2} studies absorption time for a game in which the pay-off is a random variable with range $[-\nu,+\infty)\cap \mathbb{Z}$ for a positive integer $\nu$. Common to all these is that they assume identically distributed increments, whereas we allow that these depend on the given state.

The most relevant to our present work are the following two papers. In Gut's paper \cite{gut2013gambler}, a particular instance of DGR when all the delays are the same is investigated. We recover all of his results (with slightly different notation) by setting $\gamma_1=\dots=\gamma_{n-1}=1-r$. El-Shehawey \cite{Shehawey} allows all the probabilities to win, lose or draw to depend on the player's current fortune. This is indeed a more general setting then ours but unfortunately only absorption (\ie ruin) probabilities are provided there and the expected waiting time until absorption is not considered.

We will now derive the formula for the expected time of absorption of DGR. As before, $T$ is the time of absorption. To simplify notation, for each $0\le k\le n$ we denote $\mean{T\mid X_0=k}$ by $E_k$. Note that the ratio $\frac{p_{k,k-1}}{p_{k,k+1}}=\frac{p}{1-p}\in[0,1]$ is fixed and we denote it by $\lambda$. In order to calculate the expected time of absorption, we need to solve the following recurrence relation
\begin{align}\label{eq:dgr}
\gamma_k E_k=1+\gamma_k(p E_{k-1}+(1-p)E_{k+1}),
\end{align}
for $0<k<n$, with the boundary conditions $E_0=0$, and $E_n=0$. Note that the associated homogeneous equation
\begin{align*}
E_k=p E_{k-1}+(1-p)E_{k+1}
\end{align*}
whose solutions yield the probabilities for the chain to be absorbed in $0$ or $n$, depending on which boundary conditions are imposed, is the same as in the case of classical GR. In other words, since the equation above does not depend on the lagging parameters $\gamma_k$, the probability that the gambler bankrupts before earning $\pounds n$ is the same for both DGR and GR.

It is not hard to see that for any $a,b\in\mathbb{R}$ the expression $a+b\lambda^k$ solves the homogeneous equation above and finding the solution is therefore just a matter of fitting the constants $a$ and $b$. In order to find all the solutions to \eqref{eq:dgr} it therefore suffices to find just one particular solution to it. One way to solve this is by assuming a series expansion $\sum_i a_i\lambda^i$ of the solution. After a somewhat tedious computation which we deliberately skip, one finally arrives at the solution
\begin{equation}\label{expectation}
E_k=\frac{1+\lambda}{1-\lambda}\Bigg(S_n\frac{1-\lambda^k}{1-\lambda^n}-\sum_{i=1}^{k-1}\frac{1}{\gamma_i}(1-\lambda^{k-i})\Bigg),
\text{ for }0\le k\le n,
\end{equation}
where
\begin{equation*}
S_n=\sum_{i=1}^{n-1}\frac{1}{\gamma_i}(1-\lambda^{n-i}).
\end{equation*}
The reader is invited to check that this indeed satisfies both the recurrence relation \eqref{eq:dgr} and the boundary conditions. Setting $\gamma_1=\dots=\gamma_{n-1}=1-r$ gives
\begin{gather*}
 S_n=\frac{1}{1-r}\bigg(n-\frac{1-\lambda^n}{1-\lambda}\bigg),\\
 E_k=\frac{1}{1-r}\cdot\frac{1+\lambda}{1-\lambda}\bigg(n\frac{1-\lambda^k}{1-\lambda^n}-k\bigg),
\end{gather*}
which coincides with the result in the aforementioned paper. To the best of our knowledge this is the first time that the explicit formula for the expected time of absorption for the gambler's ruin with delays appears in the literature.

Note that plugging in the values $\gamma_i=\frac{2i(n-i)}{n(n-1)}$ into \eqref{expectation} will give the explicit formula for the expected time of reaching a consensus on the complete graph assuming we start with $k$ supporters of the first (more persuasive if $p<1/2$) and $n-k$ of the second (weaker) option.
\begin{remark}
Note that all the formulae have a removable singularity at $1$ and hence are well defined by continuity at $\lambda=1$ which corresponds to $p=1/2$.
\end{remark}

\subsection{Monotonicity of the mean absorption time}

We now wish to show that as $p$ increases from $0$ to $1/2$ (or $\lambda$ from $0$ to $1$) the mean absorption time monotonically increases as well. We could try to prove that each $E_k$ is monotonic in $p$ but this clearly is not true even in the case with no delays. One can easily compute that, for example, $E_1$ when $n=3$ attains a global maximum at $\lambda=(-1 + \sqrt{3})/2$. For this reason we will be considering symmetric sums $E_k+E_{n-k}$.

Unfortunately, neither are these in general monotonic. It turns out, however, that for a fixed $0 < k < n$ the symmetric term $E_k+E_{n-k}$ is indeed increasing with $\lambda$, as long as we assume that the parameters $\gamma_i$ are symmetric, i.e.\ if $\gamma_i=\gamma_{n-i}$ for $0<i<n$. Note that for the application we have in mind this suffices, as the starting distribution of strategies over the graph is usually chosen in a way that makes it symmetric under swapping the strategies, and also the probability to sample a significant edge (which is interpreted as a delay parameter $\gamma_i$) only depends on the number of vertices currently playing one or the other strategy, and, as we noted before, is independent under swapping the two strategies.

We will first give the proof of this fact for the classical GR which immediately extends to the case where all the parameters $\gamma_i$ are the same. We will need the following lemma.
\begin{lemma}\label{lm:gr}
Let $\alpha>1$. The function
\begin{equation*}
f(\lambda)=\frac{1-\lambda}{1+\lambda}\cdot\frac{1+\lambda^\alpha}{1-\lambda^\alpha}
\end{equation*}
is a decreasing (non-negative) function of $\lambda$ on $[0,1]$.
\end{lemma}
\begin{proof}
The function is continuously differentiable on $(0,1)$ with $f'(0)=-2$, hence it suffices to show that $f'$ does not have zeros in $(0,1)$.

The zeros of $f'$, if existed, would have to satisfy the following equation
\begin{equation*}
\alpha\lambda^{\alpha-1}(1-\lambda^2)+\lambda^{2\alpha}-1=0,
\end{equation*}
or the equivalent one
\begin{equation*}
g(\lambda)=\alpha\bigg(\frac{1}{\lambda}-\lambda\bigg)+\bigg(\lambda^\alpha-\frac{1}{\lambda^\alpha}\bigg)=0.
\end{equation*}
Setting $\lambda=e^{-t}$ we get
\begin{equation*}
g(\lambda)=h(t)=2\alpha(\sinh(t) - \sinh (\alpha t))
\end{equation*}
and it suffices to show that $h(t)$ does not have zeros in $t\in(0,+\infty)$. But this is clear since $\sinh$ is increasing on $(0,+\infty)$ and hence $t<\alpha t$ implies $\sinh(t)-\sinh(\alpha t)<0$.
\end{proof}
\begin{theorem}
In the classical GR setting, the symmetrised expectation
\begin{align*}
E_k+E_{n-k}=n\cdot\frac{1+\lambda}{1-\lambda}\cdot \frac{(1-\lambda^k)(1-\lambda^{n-k})}{1-\lambda^n}
\end{align*}
is a (non-negative) increasing function of $\lambda$ on $[0,1]$ for each $0\le k \le n$.
\end{theorem}
\begin{proof}
Note,
\begin{equation*}
1-\lambda^{n}=\frac{1}{2}(1-\lambda^k)(1+\lambda^{n-k})+\frac{1}{2}(1+\lambda^k)(1-\lambda^{n-k}).
\end{equation*}
Hence,
\begin{equation*}
\frac{1}{E_k+E_{n-k}}=\frac{1}{2n}\bigg(\frac{1-\lambda}{1+\lambda}\cdot \frac{1+\lambda^{n-k}}{1-\lambda^{n-k}}+\frac{1-\lambda}{1+\lambda}\cdot \frac{1+\lambda^{k}}{1-\lambda^{k}}\bigg)
\end{equation*}
and applying Lemma \ref{lm:gr} twice yields the result.
\end{proof}

We would now like to prove the same result for the general symmetric DGR. The expression for the symmetric term is
\begin{multline}\label{eq:exp}
E_k+E_{n-k}=\frac{1+\lambda}{1-\lambda}\Bigg[\sum_{i=1}^{n-1}\frac{1}{\gamma_i}(1-\lambda^{n-i})\frac{2-\lambda^k-\lambda^{n-k}}{1-\lambda^n}\\
-\sum_{i=1}^{k-1}\frac{1}{\gamma_i}(1-\lambda^{k-i})-\sum_{i=1}^{n-k-1}\frac{1}{\gamma_i}(1-\lambda^{n-k-i})\Bigg].
\end{multline}
For a fixed $0<i<n$ letting $\gamma_i=\gamma_{n-i}$ tend to $0$ while keeping the rest of the parameters bounded away from zero, the terms containing $\frac{1}{\gamma_i}=\frac{1}{\gamma_{n-i}}$ will become dominant which means that the expression above increases with $\lambda$ if and only if each of those terms increases with $\lambda\in[0,1]$. It now remains to collect the like terms involving $\frac{1}{\gamma_i}=\frac{1}{\gamma_{n-i}}$, and to show that these are increasing with $\lambda$.

Let us fix $k$ and $i$. We may assume that $0<k,i\le n/2$ as it is assumed that $\gamma_i$'s are invariant under changing $i$ with $n-i$, and as the expression under consideration $E_k + E_{n-k}$ is also symmetric. The term multiplying $\frac{1}{\gamma_i}=\frac{1}{\gamma_{n-i}}$ in \eqref{eq:exp} is
\begin{equation}\label{eq:proof}
\frac{1+\lambda}{1-\lambda}\bigg[(2-\lambda^{n-i}-\lambda^{i})\frac{2-\lambda^k-\lambda^{n-k}}{1-\lambda^n}
-(1-\lambda^{k-i})-(1-\lambda^{n-k-i})\bigg]
\end{equation}
if $i<k$, and hence $n-i>n-k$; and
\begin{equation}\label{eq:double}
\frac{1+\lambda}{1-\lambda}\bigg[(2-\lambda^{n-i}-\lambda^{i})\frac{2-\lambda^k-\lambda^{n-k}}{1-\lambda^n}
-(1-\lambda^{n-k-i})-(1-\lambda^{i-k})\bigg]
\end{equation}
if $i>k$, and hence $n-i<n-k$. If $k=i$ both expressions are valid. Notice that swapping $i$ with $k$ transforms one into another and therefore it suffices to prove that the expression in \eqref{eq:proof} is an increasing function of $\lambda\in[0,1]$ for fixed $0<i\le k \le n/2$.
\begin{remark}
Note that in case $n$ is even and $i=n/2$, we have $i=n-i$ and also $k\le i$, so in order not to double-count, the expression we should be considering is not \eqref{eq:double} but rather
\begin{equation*}
\frac{1+\lambda}{1-\lambda}\bigg[(1-\lambda^{n/2})\frac{2-\lambda^k-\lambda^{n-k}}{1-\lambda^n}
-(1-\lambda^{n/2-k})\bigg]
\end{equation*}
which is exactly a half of \eqref{eq:double}. It therefore still suffices to show monotonicity of \eqref{eq:double}, or equivalently \eqref{eq:proof}.
\end{remark}

Let us denote by $G(\lambda)$ the expression inside the square brackets in \eqref{eq:proof}:
\begin{equation*}
G(\lambda) =(2-\lambda^{n-i}-\lambda^{i})\frac{2-\lambda^k-\lambda^{n-k}}{1-\lambda^n}
-(1-\lambda^{k-i})-(1-\lambda^{n-k-i}).
\end{equation*}
Then in order to show that $\frac{1+\lambda}{1-\lambda}\cdot G(\lambda)$ is increasing on $[0,1]$ it is sufficient to show that
\begin{equation*}
H(\lambda)=\frac{1+\lambda^i}{1-\lambda^i}\cdot G(\lambda)
\end{equation*}
is increasing and non-negative on the same domain, as by virtue of Lemma \ref{lm:gr} we know that
$\frac{1+\lambda}{1-\lambda}\cdot\frac{1-\lambda^i}{1+\lambda^i}$ is non-negative and increasing, and so will be the product of the two. We calculate,
\begin{align*}
G(\lambda)&=\frac{(2-\lambda^{n-i}-\lambda^{i})(2-\lambda^k-\lambda^{n-k})-(2-\lambda^{k-i}-\lambda^{n-k-i})(1-\lambda^n)}{1-\lambda^n}\\
&=\frac{2-2\lambda^k-2\lambda^{n-k}-2\lambda^{n-i}-2\lambda^i+\lambda^{k+i}+\lambda^{n-(k-i)}+\lambda^{k-i}+\lambda^{n-(k+i)}+2\lambda^n}{1-\lambda^n}\\
&=\frac{(1-\lambda^i)\big[2-2\lambda^{n-i}+\lambda^{k-i}+\lambda^{n-(k+i)}-\lambda^k-\lambda^{n-k}\big]}{1-\lambda^n}\\
&=\frac{(1-\lambda^i)\big[2(1-\lambda^{n-i})+(1-\lambda^i)(\lambda^{k-i}+\lambda^{n-(k+i)})\big]}{1-\lambda^n}
\end{align*}
and hence
\begin{align*}
H(\lambda)=\frac{1+\lambda^i}{1-\lambda^i}\cdot G(\lambda) =\frac{(1+\lambda^i)\big[2(1-\lambda^{n-i})+(1-\lambda^i)(\lambda^{k-i}+\lambda^{n-(k+i)})\big]}{1-\lambda^n}.
\end{align*}
After introducing a substitution $\lambda=e^{-2t}$,
\begin{align*}
F(t)&=\frac{1}{2}H(e^{-2t})=\frac{2\cosh(it)[\sinh((n-i)t) + \sinh(it)\cosh((n-2k)t)]}{\sinh(nt)}\\
&=\frac{2\cosh(it)\sinh((n-i)t) + \sinh(2it)\cosh((n-2k)t)}{\sinh(nt)}
\end{align*}
it suffices to show that $F$ is non-negative and decreasing on $[0,+\infty)$. Using addition formulae we can rearrange the numerator of the previous expression to read
\begin{align*}
&[2 \cosh(it)\sinh(nt)\cosh(it) - \sinh(nt)] + \sinh(nt)- 2\cosh(it)\sinh(it)\cosh(nt) \\
&\hspace{22em}+ \sinh(2it)\cosh((n-2k)t)\\
&=\sinh(nt)\cosh(2it) + \sinh(nt) - \sinh(2it)\cosh(nt)+ \sinh(2it)\cosh((n-2k)t)\\
&=\sinh((n-2i)t) + \sinh(2it)\cosh((n-2k)t) + \sinh(nt).
\end{align*}
Therefore
\begin{equation*}
F(t)=1+\frac{\sinh((n-2i)t) + \sinh(2it)\cosh((n-2k)t)}{\sinh(nt)}.
\end{equation*}
To make things cleaner, we introduce yet another substitution
\begin{align*}
Q(t)=F(t/n)-1&=\frac{\sinh((1-\alpha) t) + \sinh(\alpha t)\cosh((1-\beta)t)}{\sinh(t)}\\
&=\frac{2\sinh((1-\alpha) t) + \sinh((\alpha + \beta -1)t)+\sinh((1-\beta+\alpha)t)}{2\sinh(t)}
\end{align*}
where $\alpha=\frac{2i}{n}$, $\beta=\frac{2k}{n}$, and since $0<i\le k\le n/2$ we have
$0<\alpha\le \beta \le 1$. It is clear now from the formula that $F(t)\ge 1$ on $[0,\infty)$ and in particular $F$ is non-negative on the positive reals. It therefore remains to show that $Q$ is decreasing on $[0,\infty)$, or equivalently that $4\sinh^2(t)Q'(t)\le 0$ for $t\in [0,\infty)$.
We calculate
\begin{align*}
W(t)&=4\sinh^2(t)Q'(t)\\
&=2\sinh(t)\big[2(1-\alpha)\cosh((1-\alpha) t) + (\alpha + \beta -1)\cosh((\alpha + \beta -1)t)\\
&\quad+(1-\beta+\alpha)\cosh((1-\beta+\alpha)t)\big]-2\cosh(t)\big[2\sinh((1-\alpha) t)\\
&\quad+\sinh((\alpha + \beta -1)t)+\sinh((1-\beta+\alpha)t)\big]\\
&=2(1-\alpha)\big[\sinh((2-\alpha) t)+\sinh(\alpha t)\big]+(\alpha+\beta-1)\big[\sinh((\alpha+\beta)t)\\
&\quad+\sinh((2-\alpha-\beta)t)\big]+(1-\beta+\alpha)\big[\sinh((2-\beta+\alpha)t)+\sinh((\beta-\alpha)t)\big]\\
&\quad-2\big[\sinh((2-\alpha) t)-\sinh((\alpha) t)\big]-\big[\sinh((\alpha + \beta)t)-\sinh((2-\alpha - \beta)t)\big]\\
&\quad-\big[\sinh((2-\beta+\alpha)t)-\sinh((\beta-\alpha)t)\big]\\
&=-2\alpha\sinh((2-\alpha) t)+2(2-\alpha)\sinh(\alpha t)-(2-\alpha-\beta)\sinh((\alpha+\beta)t)\\
&\quad+(\alpha + \beta)\sinh((2-\alpha - \beta)t)-(\beta-\alpha)\sinh((2-\beta+\alpha)t)\\
&\quad+(2-\beta+\alpha)\sinh((\beta-\alpha)t)
\end{align*}
This last expression for $W(t)$ clearly evaluates to zero at $t=0$ and therefore it is enough to show that this is decreasing on $t\in[0,\infty)$, in other words it suffices to show $W'(t)\le 0$ for $t\ge 0$. We calculate again,
\begin{align*}
W'(t)&=2(2-\alpha)\alpha\big[\cosh(\alpha t)-\cosh((2-\alpha) t)\big]\\
&\quad+(\alpha + \beta)(2-\alpha - \beta)\big[\cosh((2-\alpha - \beta)t)-\cosh((\alpha + \beta)t)\big]\\
&\quad+(2-\beta+\alpha)(\beta-\alpha)\big[\cosh((\beta-\alpha)t)-\cosh((2-\beta+\alpha)t)\big]\\
&=4(2-\alpha)\alpha\sinh(t)\sinh((\alpha -1)t)\\
&\quad+2(\alpha + \beta)(2-\alpha - \beta)\sinh(t)\sinh((1-\alpha-\beta)t)\\
&\quad+2(2-\beta+\alpha)(\beta-\alpha)\sinh(t)\sinh((\beta-\alpha-1)t)\\
&=-\sinh(t)\big[4(2-\alpha)\alpha\sinh((1-\alpha)t)
+2(2-\beta+\alpha)(\beta-\alpha)\sinh((1-\beta+\alpha)t)\\
&\quad-2(\alpha + \beta)(2-\alpha - \beta)\sinh((1-\alpha-\beta)t)\big]
\end{align*}
In the case $\alpha+\beta > 1$ the claim easily follows as the minus sign in front of the third term can be used to change the argument of that $\sinh$ function to $(\alpha+\beta -1)t$. Recalling that $0<\alpha\le \beta\le 1$ it is easy to check that all the other constant factors appearing in the expression are non-negative.

In the case $\alpha+\beta \le 1$, the claim follows from the facts that $1-\alpha \ge 1-\alpha-\beta$, $1-\beta+\alpha \ge 1-\alpha-\beta$,
\[4(2-\alpha)\alpha+2(2-\beta+\alpha)(\beta-\alpha) \ge 2(\alpha + \beta)(2-\alpha - \beta),\]
and the following lemma.
\begin{lemma}
Let $a_1,a_2,a_3,b_1,b_2,b_3$ be non-negative real numbers such that $b_1 \ge b_3$, $b_2 \ge b_3$, and $a_1+a_2 \ge a_3$. Then for all $t\ge 0$
\begin{equation*}
a_1\sinh(b_1 t)+a_2\sinh(b_2 t)-a_3\sinh(b_3 t) \ge 0.
\end{equation*}
\end{lemma}
\begin{proof}
We rewrite the left hand side as
\begin{equation*}
a_1[\sinh(b_1 t) - \sinh (b_3 t)] + a_2[\sinh(b_2 t) - \sinh (b_3 t)] + (a_1+a_2-a_3)\sinh(b_3 t).
\end{equation*}
Since $\sinh$ is an increasing function, each of the terms above is non-negative.
\end{proof}

This completes the proof of the following theorem.
\begin{theorem}
For each $0\le k \le n$ the symmetric sum of the mean absorption times $\mean{T\mid X_0=k}+\mean{T\mid X_0=n-k}$ of gambler's ruin with symmetric delays is monotonically increasing with $p\in[0,1/2]$.
\end{theorem}
In particular, we have proved the following result.
\begin{theorem}
For $m=2$ and $G=K_n$, if the initial state is chosen symmetrically with respect to swapping strategies (\eg uniformly at random), then the expected time until reaching consensus increases monotonically with $p\in[0,1/2]$.
\end{theorem}
\section*{Acknowledgements}
The authors acknowledge support from the European Union through funding under FP7-ICT-2011-8 project HIERATIC (316705), and would like to thank the anonymous referee for their careful reading and helpful comments.

\end{document}